\documentclass[12pt,a4paper]{article}
\usepackage[ngerman, english]{babel}
\usepackage{enumerate}
\usepackage{amsmath,amssymb, amsthm, amsfonts}
\usepackage[pagebackref,colorlinks]{hyperref}

\theoremstyle{plain}
\newtheorem {lemma}{Lemma}
\newtheorem {proposition}[lemma]{Proposition}
\newtheorem {theorem}[lemma]{Theorem}

\newtheorem {corollary}[lemma]{Corollary}
\newtheorem {Key Lemma}[lemma]{Key Lemma}

\theoremstyle{definition}
\newtheorem{definition}[lemma]{Definition}

\newtheorem{remark}[lemma]{Remark}

\newtheorem {question}[lemma]{Question}

\newcommand{\Oo}{\mathcal{O}}
\newcommand{\sr}{\operatorname{sr}}
\newcommand{\m}{\mathfrak{m}}
\newcommand{\Rad}{\operatorname{Rad}}
\newcommand{\Max}{\operatorname{Max}}
\newcommand{\N}{\mathbb{N}}

\newcommand{\Z}{\mathbb{Z}}
\newcommand{\GL}{\operatorname{GL}}
\newcommand{\SL}{\operatorname{SL}}
\newcommand{\E}{\operatorname{E}}
\newcommand{\C}{\operatorname{C}}
\newcommand{\Perm}{\operatorname{P}}
\newcommand{\Cl}{\mathfrak{C}}
\newcommand{\U}{\operatorname{U}}
\newcommand{\Ort}{\operatorname{O}}
\newcommand{\Mat}{\operatorname{M}}
\newcommand{\F}{\mathbb{F}}

\title{Reverse decomposition of unipotents over noncommutative rings I: General linear groups}
\author{Raimund Preusser}
\date{}

\begin{document}
\maketitle
\begin{abstract}
\noindent
Recently is has been proved that if $\sigma\in \GL_n(R)$ where $R$ is an commutative ring and $n\geq 3$, then each of the matrices $t_{kl}(\sigma_{ij})~(i\neq j,k\neq l)$ is a product of eight $\E_n(R)$-conjugates of $\sigma$ and $\sigma^{-1}$. In this article we show that similar results hold true if $R$ is a (noncommutative) von Neumann regular ring, or a Banach algebra, or a ring satisfying a stable range condition, or a ring with Euclidean algorithm, or an almost commutative ring.
\end{abstract}
\let\thefootnote\relax
\let\thefootnote\relax

\section{Introduction}
In 1964, H.\ Bass \cite{bass} showed that if $R$ is a ring and $H$ a subgroup of the general linear group $\GL_n(R)$, then 
\begin{equation*}
H\text{ is normalised by }\E_n(R)~\Leftrightarrow~ \exists \text{ ideal }I:~\E_n(R,I)\subseteq H\subseteq \C_n(R,I)
\end{equation*}
provided $n$ is large enough with respect to the stable rank of $R$. Here $\E_n(R)$ denotes the elementary subgroup, $\E_n(R,I)$ the relative elementary subgroup of level $I$ and $\C_n(R,I)$ the full congruence subgroup of level $I$ (cf. \cite{preusser2}). Bass's result, which is one of the central points in the structure theory of general linear groups, is known as {\it Sandwich Classification Theorem}. In the 1970's and 80's, the validity of this theorem was extended by J.\ Wilson \cite{wilson}, I.\ Golubchik \cite{golubchik}, L.\ Vaserstein \cite{vaserstein, vaserstein_ban, vaserstein_neum} and others. It holds true, for example, if $R$ is almost commutative (i.e. finitely generated as a module over its center) and $n\geq 3$.

It follows from the Sandwich Classification Theorem that if $R$ is a commutative ring, $n\geq 3$, $\sigma\in \GL_n(R)$, $i\neq j$ and $k\neq l$, then the elementary transvection $t_{kl}(\sigma_{ij})$ can be expressed as a finite product of $\E_n(R)$-conjugates of $\sigma$ and $\sigma^{-1}$. 
In 1960, J.\ Brenner \cite{brenner} showed that if $R=\Z$, then there is a bound for the number of factors needed for such an expression of $t_{kl}(\sigma_{ij})$, but until recently there was not much hope to prove such a result for arbitrary commutative rings $R$. 

However, in 2018 the author found an explicit expression of $t_{kl}(\sigma_{ij})$ (where $\sigma\in \GL_n(R)$, $R$ is a commutative ring, $n\geq 3$, $i\neq j$, $k\neq l$) as a product of $8$ elementary conjugates of $\sigma$ and $\sigma^{-1}$, yielding a very short proof of the Sandwich Classification Theorem for commutative rings \cite{preusser2}. Similar results were obtained for the even- and odd-dimensional orthogonal groups $\Ort_{2n}(R)$ and $\Ort_{2n+1}(R)$, and the even- and odd-dimensional unitary groups $\U_{2n}(R,\Lambda)$ and $\U_{2n+1}(R,\Delta)$ where $R$ is a commutative ring and $n\geq 3$ \cite{preusser2, preusser3}.

In order to show the normality of the elementary subgroup $\E_n(R)$ in $\GL_n(R)$, one has to prove that $\E_n(R)^\sigma\subseteq \E_n(R)$ for any $\sigma\in\GL_n(R)$. {\it Decomposition of unipotents} provides explicit formulae expressing the generators $t_{kl}(x)^\sigma~(k\neq l, x\in R)$ of $\E_n(R)^{\sigma}$ as products of elementary transvections, see \cite{stepanov-vavilov}. In order to show the Sandwich Classification Theorem for commutative rings, one has to prove that $\E_n(I)\subseteq \sigma^{E_n(R)}$ for any $\sigma\in\GL_n(R)$ where $I$ denotes the ideal generated by the nondiagonal entries of $\sigma$ and all differences of diagonal entries. The paper \cite{preusser2} provides explicit formulae expressing the generators $t_{kl}(x\sigma_{ij}),t_{kl}(x(\sigma_{ii}-\sigma_{jj}))~(k\neq l,i\neq j,x\in R)$ of $\E_n(I)$, as products of $\E_n(R)$-conjugates of $\sigma$ and $\sigma^{-1}$. N.\ Vavilov considered the papers \cite{preusser2, preusser3} ``the first major advance in the direction of what can be dubbed the {\it reverse decomposition of unipotents}'', see \cite{vavilov}. 

In this paper we obtain bounds for the reverse decomposition of unipotents over different classes of noncommutative rings. 
Pick a $\sigma\in \GL_n(R)$. Our goal is to express the matrix $t_{kl}(\sigma_{ij})$ where $k\neq l$ and $i\neq j$ as a product of $E_n(R)$-conjugates of $\sigma$ and $\sigma^{-1}$. An idea is to start with $\sigma$ and then apply successively operations of types (a) and (b) below until one arrives at $t_{kl}(\sigma_{ij})$.
\begin{enumerate}[(a)]
\item $\GL_n(R)\rightarrow \GL_n(R),\tau\mapsto \tau^\xi$ where $\xi\in \E_n(R)$.
\item $\GL_n(R)\rightarrow \GL_n(R),\tau\mapsto [\tau,\xi]$ where $\xi\in \E_n(R)$.
\end{enumerate} 
If $\tau$ lies in $\sigma^{\E_n(R)}$ and $\rho$ is obtained from $\tau$ by applying an operation of type (a) or (b), then clearly $\rho$ again lies in $\sigma^{\E_n(R)}$.

In this paper we also use operations of type (c) below. 
\begin{enumerate}[(c)]
\setcounter{enumi}{3}
\item $\GL_n(R)\times\GL_n(R)\rightarrow \GL_n(R)\times\GL_n(R),(\tau_1,\tau_2)\mapsto ([\tau_1^{-1},\xi],[\xi,\tau_2])$ where $\xi\in \E_n(R)$.
\end{enumerate} 
We show in Section 4 that if the product $\tau_1\tau_2$ lies in $\sigma^{\E_n(R)}$ where $\tau_1\in\E_n(R)$, and $(\rho_1, \rho_2)$ is obtained from $(\tau_1,\tau_2)$ by applying an operation of type (c), then the product $\rho_1\rho_2$ again lies in $\sigma^{\E_n(R)}$.

The rest of the paper is organised as follows. In Section 2 we recall some standard notation which is used throughout the paper. In Section 3 we recall the definitions of the general linear group $\GL_n(R)$ and its elementary subgroup $\E_n(R)$. In Section 4, we introduce the operations $(c)$ above in the abstract context of groups. Then we use these operations to obtain results for general linear groups which are crucial for the following sections. In Section 5 we show how \cite[Theorem 12]{preusser2} can be derived from the results of Section 4. In Sections 6-10 we get bounds for the reverse decomposition of unipotents over von Neumann regular rings, Banach algebras, rings satisfying a stable range condition, rings with Euclidean algorithm and almost commutative rings. In the last section we list some open problems.

\section{Notation}
$\N$ denotes the set of positive integers. If $G$ is a group and $g,h\in G$, we let $g^h:=h^{-1}gh$, and $[g,h]:=ghg^{-1}h^{-1}$. If $g\in G$ and $H$ is a subgroup of $G$, we denote by $g^{H}$ the subgroup of $G$ generated by the set $\{g^h\mid h\in H\}$. By a ring  we mean an associative ring with $1\neq 0$. By an ideal of a ring we mean a twosided ideal.

Throughout the paper $R$ denotes a ring and $n$ a positive integer greater than $2$. We denote by $^n\!R$ the set of all rows $u=(u_1,\dots,u_n)$ with entries in $R$ and by $R^n$ the set of all columns $v=(v_1,\dots,v_n)^t$ with entries in $R$. Furthermore, the set of all $n\times n$ matrices over $R$ is denoted by $\Mat_n(R)$. The identity matrix in $\Mat_n(R)$ is denoted by $e$ or $e_{n\times n}$ and the matrix with a one at position $(i,j)$ and zeros elsewhere is denoted by $e^{ij}$. If $\sigma\in \Mat_{n
}(R)$, we denote the 
the entry of $\sigma$ at position $(i,j)$ by $\sigma_{ij}$. We denote the $i$-th row of $\sigma$ by $\sigma_{i*}$ and its $j$-th column by $\sigma_{*j}$. If $\sigma\in \Mat_n(R)$ is invertible, we denote the entry 
of $\sigma^{-1}$ at position $(i,j)$ by $\sigma'_{ij}$, the $i$-th row of $\sigma^{-1}$ by $\sigma'_{i*}$ and the $j$-th column of $\sigma^{-1}$ by $\sigma'_{*j}$.  

\section{The general linear group and its elementary subgroup} 

\begin{definition}
The group $\GL_{n}(R)$ consisting of all invertible elements of $\Mat_n(R)$ is called the {\it general linear group} of degree $n$ over $R$.
\end{definition}

\begin{definition}
Let $x\in R$ and $i,j\in\{1,\dots,n\}$ such that $i\neq j$. Then the matrix $t_{ij}(x):=e+xe^{ij}$ is called an {\it elementary transvection}. The subgroup $\E_n(R)$ of $\GL_n(R)$ generated by the elementary transvections is called the {\it elementary subgroup}.
\end{definition}

\begin{lemma}\label{lemelrel}
The relations
\begin{align*}
t_{ij}(x)t_{ij}(y)&=t_{ij}(x+y), \tag{R1}\\
[t_{ij}(x),t_{hk}(y)]&=e \tag{R2}\text{ and}\\
[t_{ij}(x),t_{jk}(y)]&=t_{ik}(xy) \tag{R3}\\
\end{align*}
hold where $i\neq k, j\neq h$ in $(R2)$ and $i\neq k$ in $(R3)$.
\end{lemma}
\begin{proof}
Straightforward computation.
\end{proof}

\begin{definition}\label{defp}
Let $i,j\in\{1,\dots,n\}$ such that $i\neq j$. Then the matrix $p_{ij}:=e+e^{ij}-e^{ji}-e^{ii}-e^{jj}=t_{ij}(1)t_{ji}(-1)t_{ij}(1)\in \E_{n}(R)$ is called a {\it generalised permutation matrix}. It is easy to show that $p_{ij}^{-1}=p_{ji}$. The subgroup of $\E_n(R)$ generated by the generalised permutation matrices is denoted by $\Perm_n(R)$.
\end{definition}

\begin{lemma}\label{lemp}
Let $\sigma\in \GL_n(R)$. Further let $x\in R$ and $i,j,k,l\in\{1,\dots,n\}$ such that $i\neq j$ and $k\neq l$. Then there are $\tau,\rho\in \Perm_n(R)$ such that $(\sigma^{\tau})_{kl}=\sigma_{ij}$ and $t_{kl}(x)^{\rho}=t_{ij}(x)$.
\end{lemma}
\begin{proof}
Easy exercise.
\end{proof}

\section{The key results}
\subsection{Simultaneous reduction in groups}

In this subsection $G$ denotes a group. 
\begin{definition}
Let $(a_1,b_1), (a_2,b_2)\in G\times G$. If there is an $g\in G$ such that 
\[a_2=[a_1^{-1},g]\text{ and }b_2=[g,b_1],\]
then we write $(a_1,b_1)\xrightarrow{g} (a_2,b_2)$. In this case $(a_1,b_1)$ is called {\it reducible to $(a_2,b_2)$ by $g$}. 
\end{definition}

\begin{definition}
If $(a_1,b_1),\dots, (a_{n+1},b_{n+1})\in G\times G$ and $g_1,\dots,g_{n}\in G$  such that 
\[(a_1,b_1) \xrightarrow{g_1} (a_2,b_2)\xrightarrow{g_2}\dots \xrightarrow{g_{n}} (a_{n+1},b_{n+1}),\]
then we write $(a_1,b_1) \xrightarrow{g_1,\dots,g_{n}}(a_{n+1},b_{n+1})$. In this case $(a_1,b_1)$ is called {\it reducible to $(a_{n+1},b_{n+1})$ by $g_1,\dots,g_n$}.
\end{definition}

Let $H$ be a subgroup of $G$. If $g\in G$ and $h\in H$, then we call $g^h$ an {\it $H$-conjugate} of $g$.
\begin{lemma}\label{lemredux}
Let $(a_1,b_1),(a_2,b_2)\in G\times G$. If $(a_1,b_1)\xrightarrow{g_1,\dots,g_n}(a_2,b_2)$ for some $g_1,\dots,g_n\in G$, then $a_2b_2$ is a product of $2^n$ $H$-conjugates of $a_1b_1$ and $(a_1b_1)^{-1}$ where $H$ is the subgroup of $G$ generated by $\{a_1,g_1,\dots, g_n\}$.
\end{lemma}
\begin{proof}
Assume that $n=1$. Then 
\[a_2b_2=[a_1^{-1},g_1][g_1,b_1]=(a_1b_1)^{g_1^{-1}a_1}\cdot((a_1b_1)^{-1})^{a_1}.\]
Hence $a_2b_2$ is a product of two $H$-conjugates of $a_1b_1$ and $(a_1b_1)^{-1}$. The general case follows by induction (note that $a_1,\dots,a_n\in H$).
\end{proof}

\subsection{Application to general linear groups}

\begin{proposition}\label{propkey}
Let $\sigma\in \GL_n(R)$. Let $x_1,\dots,x_{n},y\in R$ such that $y\sum\limits_{p=1}^{n}\sigma_{1p}x_p=0$. Then the following holds. 
\begin{enumerate}[(i)]
\itemsep0pt 
\item Suppose that $x_n=0$, $y=1$. Then for any $k\neq l$ and $a,b\in R$, the elementary transvection $t_{kl}(ax_1b)$ is a product of $8$ $\E_n(R)$-conjugates of $\sigma$ and $\sigma^{-1}$.
\item Suppose that $x_n=1$, $y=1$. Then for any $k\neq l$ and $a,b\in R$, the elementary transvection $t_{kl}(ax_1b)$ is a product of $8$ $\E_n(R)$-conjugates of $\sigma$ and $\sigma^{-1}$.
\item Suppose that $x_1=1$. Then for any $k\neq l$ and $a,b\in R$, the elementary transvection $t_{kl}(ayb)$ is a product of $8$ $\E_n(R)$-conjugates of $\sigma$ and $\sigma^{-1}$.
\end{enumerate}
\end{proposition}
\begin{proof}
\begin{enumerate}[(i)]
\item Set $\tau:=\prod\limits_{p=1}^{n-1}t_{pn}(x_p)$. Clearly $(\sigma\tau^{-1})_{1*}=\sigma_{1*}$ and hence $(\sigma\tau^{-1}\sigma^{-1})_{1*}=e_{1*}$. A straightforward computation shows
\[(\tau,\sigma\tau^{-1}\sigma^{-1})\xrightarrow{t_{21}(a),t_{n1}(-b)}(t_{21}(ax_1b), e).\] 
It follows from Lemma \ref{lemredux} that $t_{21}(ax_1b)$ is a product of $4$ $\E_n(R)$-conjugates of $[\tau,\sigma]$ and $[\tau,\sigma]^{-1}$. Thus $t_{21}(ax_1b)$ is a product of $8$ $\E_n(R)$-conjugates of $\sigma$ and $\sigma^{-1}$. Assertion (i) follows now from Lemma \ref{lemp}.
\item Set $\tau:=\prod\limits_{p=1}^{n-1}t_{pn}(x_p)$. Clearly $(\sigma\tau)_{1n}=0$. A straightforward computation shows that
\[(\tau,\tau^{-1}\sigma^{-1})\xrightarrow{t_{21}(a),t_{n1}(-b),t_{n2}(1)}(t_{n1}(ax_1b), e).\] 
It follows from Lemma \ref{lemredux} that $t_{n1}(ax_1b)$ is a product of $8$ $\E_n(R)$-conjugates of $\sigma$ and $\sigma^{-1}$. Assertion (ii) follows now from Lemma \ref{lemp}.
\item Clearly $t_{n1}(y)^{\sigma}=e+\sigma'_{*n}y\sigma_{1*}$. Set $\tau:=\prod\limits_{p=2}^{n}t_{p1}(x_p)$. One checks easily that $(t_{n1}(y)^{\sigma\tau})_{*1}=e_{*1}$. A straightforward computation shows that
\[(t_{n1}(-y),t_{n1}(y)^{\sigma\tau})\xrightarrow{t_{12}(b),t_{1n}(-a)}(t_{12}(ayb), e).\] 
It follows from Lemma \ref{lemredux} that $t_{12}(ayb)$ is a product of $4$ $\E_n(R)$-conjugates of $t_{n1}(-y)t_{n1}(y)^{\sigma\tau}=(\sigma^{-1})^{\tau t_{n1}(y)}\cdot \sigma^\tau$ and its inverse. Thus $t_{12}(ayb)$ is a product of $8$ $\E_n(R)$-conjugates of $\sigma$ and $\sigma^{-1}$. Assertion (iii) follows now from Lemma \ref{lemp}.
\end{enumerate}
\end{proof}

\begin{corollary}\label{corkeyB}
Let $\sigma\in \GL_n(R)$ such that $\sigma_{1n}=0$. Then for any $2\leq j\leq n$, $k\neq l$ and $a,b\in R$, the elementary transvection $t_{kl}(a\sigma'_{1j}b)$ is a product of $8$ $\E_n(R)$-conjugates of $\sigma$ and $\sigma^{-1}$.
\end{corollary}
\begin{proof}
Clearly we have $\sum\limits_{p=1}^{n-1}\sigma_{1p}\sigma'_{pj}=0$ for any $2\leq j\leq n$. Hence, by Proposition \ref{propkey}(i), the elementary transvection $t_{kl}(a\sigma'_{1j}b)$ is a product of $8$ $\E_n(R)$-conjugates of $\sigma$ and $\sigma^{-1}$.
\end{proof}

\begin{corollary}\label{corkeyA}
Let $\sigma\in \GL_n(R)$ such that $\sigma_{11}$ is right invertible. Then for any $k\neq l$ and $a,b\in R$, the elementary transvection $t_{kl}(a\sigma_{1n}b)$ is a product of $8$ $\E_n(R)$-conjugates of $\sigma$ and $\sigma^{-1}$.
\end{corollary}
\begin{proof}
Let $z$ be a right inverse of $\sigma_{11}$. Then $-\sigma_{11}z\sigma_{1n}+\sigma_{1n}=0$. Hence, by Proposition \ref{propkey}(ii), the elementary transvection $t_{kl}(a\sigma_{1n}b)$ is a product of $8$ $\E_n(R)$-conjugates of $\sigma$ and $\sigma^{-1}$.
\end{proof}

\begin{corollary}\label{corkeyC}
Let $\sigma\in \GL_n(R)$ such that $\sigma_{1n}$ is an idempotent. Then for any $k\neq l$ and $a,b\in R$, the elementary transvection $t_{kl}(a\sigma_{1n}b)$ is a product of $8$ $\E_n(R)$-conjugates of $\sigma$ and $\sigma^{-1}$. In particular, if $\sigma_{1n}=1$, then  $\E_n(R)\subseteq \sigma^{\E_n(R)}$.
\end{corollary}
\begin{proof}
Clearly we have $\sigma_{1n}(\sigma_{11}-\sigma_{1n}\sigma_{11})=0$. Hence, by Proposition \ref{propkey}(iii), the elementary transvection $t_{kl}(a\sigma_{1n}b)$ is a product of $8$ $\E_n(R)$-conjugates of $\sigma$ and $\sigma^{-1}$.
\end{proof}

\begin{corollary}\label{corkeyD}
Let $\sigma\in \GL_n(R)$. Let $x_1,\dots,x_{n-1}\in R$ such that $\sum\limits_{p=1}^{n-1}\sigma_{1p}x_p+\sigma_{1n}=0$. Then for any $2\leq j\leq n$, $k\neq l$ and $a,b\in R$, the elementary transvection $t_{kl}(a\sigma'_{1j}b)$ is a product of $16$ $\E_n(R)$-conjugates of $\sigma$ and $\sigma^{-1}$.
\end{corollary}
\begin{proof}
One checks easily that $\sum\limits_{p=1}^{n-1}\sigma_{1p}(\sigma'_{pj}-x_p\sigma'_{nj})=0$. Hence, by Proposition \ref{propkey}(i), the elementary transvection $t_{kl}(a(\sigma'_{1j}-x_1\sigma'_{nj})b)$ is a product of $8$ $\E_n(R)$-conjugates of $\sigma$ and $\sigma^{-1}$. By Proposition \ref{propkey}(ii), $t_{kl}(ax_1\sigma'_{nj}b)$ is a product of $8$ $\E_n(R)$-conjugates of $\sigma$ and $\sigma^{-1}$. Thus $t_{kl}(a\sigma'_{1j}b)$ is a product of $16$ $\E_n(R)$-conjugates of $\sigma$ and $\sigma^{-1}$.
\end{proof}

\section{RDU over commutative rings}
\begin{theorem}[cf. {\cite[Theorem 12]{preusser2}}]\label{thmcomm}
Suppose that $R$ is commutative. Let $\sigma\in \GL_n(R)$, $i\neq j$, $k\neq l$ and $a\in R$. Then 
\begin{enumerate}[(i)]
\itemsep0pt 
\item $t_{kl}(a\sigma_{ij})$ is a product of $8$ $\E_n(R)$-conjugates of $\sigma$ and $\sigma^{-1}$ and
\item $t_{kl}(a(\sigma_{ii}-\sigma_{jj}))$ is a product of $24$ $\E_n(R)$-conjugates of $\sigma$ and $\sigma^{-1}$.
\end{enumerate}
\end{theorem}
\begin{proof}
\begin{enumerate}[(i)]
\item Clearly $\sigma_{11}\sigma_{12}-\sigma_{12}\sigma_{11}=0$. Hence, by Proposition \ref{propkey}(i), $t_{kl}(a\sigma_{12})$ is a product of $8$ $\E_n(R)$-conjugates of $\sigma$ and $\sigma^{-1}$. Assertion (i) now follows from Lemma \ref{lemp}.
\item Clearly the entry of $\sigma^{t_{ji}(-1)}$ at position $(j,i)$ equals $\sigma_{ii}-\sigma_{jj}+\sigma_{ji}-\sigma_{ij}$. Applying (i) to $\sigma^{t_{ji}(-1)}$ we get that $t_{kl}(a(\sigma_{ii}-\sigma_{jj}+\sigma_{ji}-\sigma_{ij}))$ is a product of $8$ $\E_n(R)$-conjugates of $\sigma$ and $\sigma^{-1}$. Applying (i) to $\sigma$ we get that $t_{kl}(a(\sigma_{ij}-\sigma_{ji}))=t_{kl}(a\sigma_{ij})t_{kl}(-a\sigma_{ji})$ is a product of $8+8=16$ $\E_n(R)$-conjugates of $\sigma$ and $\sigma^{-1}$. It follows that $t_{kl}(a(\sigma_{ii}-\sigma_{jj}))=t_{kl}(a(\sigma_{ii}-\sigma_{jj}+\sigma_{ji}-\sigma_{ij}))t_{kl}(a(\sigma_{ij}-\sigma_{ji}))$ is a product of $24$ $\E_n(R)$-conjugates of $\sigma$ and $\sigma^{-1}$.
\end{enumerate}
\end{proof}

\section{RDU over von Neumann regular rings}

Recall that $R$ is called {\it von Neumann regular}, if for any $x\in R$ there is a $y\in R$ such that $xyx=x$.
\begin{theorem}\label{thmNeum}
Suppose that $R$ is von Neumann regular. Let $\sigma\in \GL_n(R)$, $i\neq j$, $k\neq l$ and $a,b,c\in R$. Then 
\begin{enumerate}[(i)]
\itemsep0pt 
\item $t_{kl}(a\sigma_{ij}b)$ is a product of $8$ $\E_n(R)$-conjugates of $\sigma$ and $\sigma^{-1}$ and
\item $t_{kl}(a(c\sigma_{ii}-\sigma_{jj}c)b)$ is a product of $24$ $\E_n(R)$-conjugates of $\sigma$ and $\sigma^{-1}$.
\end{enumerate}
\end{theorem}
\begin{proof}
\begin{enumerate}[(i)]
\item Choose a $z$ such that $\sigma_{12}z\sigma_{12}=\sigma_{12}$. Then $\sigma_{12}z(\sigma_{11}-\sigma_{12}z\sigma_{11})=0$. By Proposition \ref{propkey}(iii) (applied with $y=\sigma_{12}z$, $x_1=1$, $x_2=-z\sigma_{11}$ and $x_3,\dots,x_n=0$), we get that $t_{kl}(a\sigma_{12}b)$ is a product of $8$ $\E_n(R)$-conjugates of $\sigma$ and $\sigma^{-1}$. Assertion (i) now follows from Lemma \ref{lemp}.
\item Clearly the entry of $\sigma^{t_{ji}(-c)}$ at position $(j,i)$ equals $c\sigma_{ii}-\sigma_{jj}c+\sigma_{ji}-c\sigma_{ij}c$. Applying (i) to $\sigma^{t_{ji}(-c)}$ we get that $t_{kl}(a(c\sigma_{ii}-\sigma_{jj}c+\sigma_{ji}-c\sigma_{ij}c)b)$ is a product of $8$ $\E_n(R)$-conjugates of $\sigma$ and $\sigma^{-1}$. Applying (i) to $\sigma$ we get that $t_{kl}(a(c\sigma_{ij}c-\sigma_{ji})b)=t_{kl}(ac\sigma_{ij}cb)t_{kl}(-a\sigma_{ji}b)$ is a product of $8+8=16$ $\E_n(R)$-conjugates of $\sigma$ and $\sigma^{-1}$. It follows that $t_{kl}(a(c\sigma_{ii}-\sigma_{jj}c)b)=t_{kl}(a(c\sigma_{ii}-\sigma_{jj}c+\sigma_{ji}-c\sigma_{ij}c)b)t_{kl}(a(c\sigma_{ij}c-\sigma_{ji})b)$ is a product of $24$ $\E_n(R)$-conjugates of $\sigma$ and $\sigma^{-1}$.
\end{enumerate}
\end{proof}

\section{RDU over Banach algebras}
Recall that a {\it Banach algebra} is an algebra $R$ over the real or the complex numbers that at the same time is also a {\it Banach space}, i.e. a normed vector space that is complete with respect to the metric induced by the norm. The norm is required to satisfy $\|x\,y\|\ \leq \|x\|\,\|y\|$ for any $x,y\in R$.

It follows from \cite[Chapter VII, Lemma 2.1]{conway_book} that any Banach algebra $R$ has the property (1) below where $R^*$ denotes the set of all right invertible elements of $R$.
\begin{equation}
\text{For any }x,z\in R\text{ there is a }y\in R^*\text{ such that } 1+xyz\in R^*.
\end{equation}

\begin{theorem}\label{thmBan}
Suppose that $R$ is a ring satisfying (1) (which is true e.g. if $R$ is a Banach algebra). Let $\sigma\in \GL_n(R)$, $i\neq j$, $k\neq l$ and $a,b,c\in R$. Then 
\begin{enumerate}[(i)]
\itemsep0pt 
\item $t_{kl}(a\sigma_{ij}b)$ is a product of $160$ $\E_n(R)$-conjugates of $\sigma$ and $\sigma^{-1}$ and
\item $t_{kl}(a(c\sigma_{ii}-\sigma_{jj}c)b)$ is a product of $480$ $\E_n(R)$-conjugates of $\sigma$ and $\sigma^{-1}$.
\end{enumerate}
\end{theorem}
\begin{proof}
\begin{enumerate}[(i)]
\item {\bf Step 1} Let $x\in R$ and set $\tau:=[\sigma,t_{1n}(x)]$. Then $\tau_{11}=1+\sigma_{11}x\sigma'_{n1}\in R^*$ for an appropriate $x\in R^*$. Let $x^{-1}$ denote a right inverse of $x$. It follows from Corollary \ref{corkeyA} that for any $k'\neq l'$ and $a',b'\in R$, the elementary transvection 
\[t_{k'l'}(a'\tau_{1n}b')=t_{k'l'}(a'(\sigma_{11}x\sigma'_{nn}-(1+\sigma_{11}x\sigma'_{n1})x)b')=t_{k'l'}(a'(\sigma_{11}\alpha-1)xb'),\] 
where $\alpha=x\sigma'_{nn}x^{-1}-x\sigma'_{n1}$, is a product of $8$ $\E_n(R)$-conjugates of $\tau$ and $\tau^{-1}$. Hence
\begin{equation}
t_{k'l'}(a'(\sigma_{11}\alpha-1)b')\text{ is a product of }16~\E_n(R)\text{-conjugates of }\sigma \text{ and }\sigma^{-1}.
\end{equation}
{\bf Step 2} Set $\zeta:=\sigma t_{1n}(-\alpha\sigma_{1n})$ where $\alpha$ is defined as in Step 1. Clearly 
\[(t_{1n}(\alpha\sigma_{1n}),\zeta)\xrightarrow{t_{n2}(y)}(t_{12}(-\alpha\sigma_{1n}y),\xi)\]
for any $y\in R$, where $\xi=[t_{n2}(y),\zeta]$. It follows from Lemma \ref{lemredux} that 
\begin{equation}
t_{12}(-\alpha\sigma_{1n}y)\xi\text{ is a product of }2~\E_n(R)\text{-conjugates of }\sigma \text{ and }\sigma^{-1}.
\end{equation}
Choose a $y\in R^*$ such that $\xi_{11}=1-\zeta_{1n}y\zeta'_{21}\in R^*$. Let $\xi_{11}^{-1}$ denote a right inverse of $\xi_{11}$. Then $\rho:=\xi t_{12}(-\xi_{11}^{-1}\xi_{12})$ has the property that $\rho_{12}=0$. Clearly $\xi_{12}=-\zeta_{1n}y\zeta'_{22}=-(1-\sigma_{11}\alpha)\sigma_{1n}y\zeta'_{22}$. Hence, by (2), 
\begin{equation}
t_{12}(-\xi_{11}^{-1}\xi_{12})\text{ is a product of }16~\E_n(R)\text{-conjugates of }\sigma \text{ and }\sigma^{-1}.
\end{equation}
It follows from (3) and (4) that 
\begin{equation}
t_{12}(-\alpha\sigma_{1n}y)\rho\text{ is a product of }18~\E_n(R)\text{-conjugates of }\sigma \text{ and }\sigma^{-1}.
\end{equation}
{\bf Step 3} Let $y^{-1}$ denote a right inverse of $y$. One checks easily that for any $a'',b''\in R$ we have 
\[(t_{12}(-\alpha\sigma_{1n}y),\rho)\xrightarrow{t_{23}(y^{-1}b''),t_{21}(a''),t_{31}(-1)}(t_{21}(a''\alpha\sigma_{1n}b''), e)\]
(recall from Step 2 that $\rho_{12}=0$). It follows from Lemma \ref{lemredux} that $t_{21}(a''\alpha\sigma_{1n}b'')$ is a product of $8$ $\E_n(R)$-conjugates of $t_{12}(-\alpha\sigma_{1n}y)\rho$ and its inverse. Hence, by (5), 
\begin{equation}
t_{21}(a''\alpha\sigma_{1n}b'')\text{ is a product of }8\cdot 18=144~\E_n(R)\text{-conjugates of }\sigma \text{ and }\sigma^{-1}.
\end{equation}
It follows from (2) and (6) that $t_{21}(a\sigma_{1n}b)=t_{21}(a(1-\sigma_{11}\alpha)\sigma_{1n}b)t_{21}(a\sigma_{11}\alpha\sigma_{1n}b)$ is a product of $144+16=160$ $\E_n(R)$-conjugates of $\sigma$ and $\sigma^{-1}$. Assertion (i) now follows from Lemma \ref{lemp}.
\item See the proof of Theorem \ref{thmNeum}.
\end{enumerate}
\end{proof}

\section{RDU in the stable range}

Recall that a row vector $u\in {}^m\!R$ is called {\it unimodular} if there is a column vector $v\in R^m$ such that $uv=1$. The {\it stable rank $\sr(R)$} is the least $m\in\N$ such that for any $p\geq m$ and unimodular row $(u_1,\dots,u_{p+1})\in {}^{p+1}\!R$ there are elements $x_1,\dots,x_p\in R$ such that $(u_1+u_{p+1}x_1,\dots,u_p+u_{p+1}x_p)\in {}^p\!R$ is unimodular (if no such $m$ exists, then $\sr(R)=\infty$).

Define $E^*_n(R):=\langle t_{kl}(x)\mid x\in R,k\neq l, k\neq 1, l\neq n\rangle$.

\begin{lemma}\label{lemsr}
Suppose that $\sr(R)<n$ and let $\sigma\in \GL_n(R)$. Then there is a $\rho\in E^*_n(R)$ such that the row vector $((\sigma^{\rho})_{11},\dots,(\sigma^{\rho})_{1,\sr(R)})$ is unimodular.
\end{lemma}
\begin{proof}
Since $\sr(R)<n$, there is a $\rho$ of the form
\[\rho=(\prod\limits_{j=1}^{n-1}t_{nj}(*))(\prod\limits_{j=1}^{n-2}t_{n-1,j}(*))\dots (\prod\limits_{j=1}^{\sr(R)}t_{\sr(R)+1,j}(*))\in E^*_n(R)\]
such that $((\sigma\rho)_{11},\dots,(\sigma\rho)_{1,\sr(R)})$ is unimodular. Clearly $((\sigma^{\rho})_{11},\dots,(\sigma^{\rho})_{1,\sr(R)})=((\sigma\rho)_{11},\dots,(\sigma\rho)_{1,\sr(R)})$ since $\rho\in E^*_n(R)$.
\end{proof}

\begin{theorem}\label{thmsr=1}
Suppose that $\sr(R)=1$. Let $\sigma\in \GL_n(R)$, $i\neq j$, $k\neq l$ and $a,b,c\in R$. Then 
\begin{enumerate}[(i)]
\itemsep0pt 
\item $t_{kl}(a\sigma_{ij}b)$ is a product of $8$ $\E_n(R)$-conjugates of $\sigma$ and $\sigma^{-1}$ and
\item $t_{kl}(a(c\sigma_{ii}-\sigma_{jj}c)b)$ is a product of $24$ $\E_n(R)$-conjugates of $\sigma$ and $\sigma^{-1}$.
\end{enumerate}
\end{theorem}
\begin{proof}
\begin{enumerate}[(i)]
\item By Lemma \ref{lemsr} there there is a $\rho\in E^*_n(R)$ such that $\hat\sigma_{11}$ is right invertible where $\hat\sigma=\sigma^{\rho}$. It follows from Corollary \ref{corkeyA} that $t_{kl}(a\hat\sigma_{1n}b)$ is a product of $8$ $\E_n(R)$-conjugates of $\hat\sigma$ and $\hat\sigma^{-1}$. Clearly $\hat\sigma_{1n}=\sigma_{1n}$ since $\rho\in E^*_n(R)$. Moreover, any $\E_n(R)$-conjugate of $\hat\sigma$ or $\hat\sigma^{-1}$ is an $\E_n(R)$-conjugate of $\sigma$ or $\sigma^{-1}$. Hence $t_{kl}(a\sigma_{1n}b)$ is a product of $8$ $\E_n(R)$-conjugates of $\sigma$ and $\sigma^{-1}$. Assertion (i) now follows from Lemma \ref{lemp}.
\item See the proof of Theorem \ref{thmNeum}.
\end{enumerate}
\end{proof}

\begin{theorem}
Suppose that $1<\sr(R)<n$. Let $\sigma\in \GL_n(R)$, $i\neq j$, $k\neq l$ and $a,b,c\in R$. Then 
\begin{enumerate}[(i)]
\itemsep0pt 
\item $t_{kl}(a\sigma_{ij}b)$ is a product of $16$ $\E_n(R)$-conjugates of $\sigma$ and $\sigma^{-1}$ and
\item $t_{kl}(a(c\sigma_{ii}-\sigma_{jj}c)b)$ is a product of $48$ $\E_n(R)$-conjugates of $\sigma$ and $\sigma^{-1}$.
\end{enumerate}
\end{theorem}
\begin{proof}
\begin{enumerate}[(i)]
\item By Lemma \ref{lemsr} there there is a $\rho\in E^*_n(R)$ such that $(\hat\sigma_{11},\dots,\hat\sigma_{1,\sr(R)})$ is unimodular where $\hat\sigma=\sigma^{\rho}$.  Clearly there are $x_1,\dots,x_{\sr(R)}\in R$ such that $\sum\limits_{p=1}^{\sr(R)}\hat\sigma_{1p}x_p+\hat\sigma_{1n}=0$. It follows from Corollary \ref{corkeyD} that $t_{kl}(a\hat\sigma'_{1n}b)$ is a product of $16$ $\E_n(R)$-conjugates of $\hat\sigma$ and $\hat\sigma^{-1}$. Clearly $\hat\sigma'_{1n}=\sigma'_{1n}$ since $\rho\in E^*_n(R)$. Moreover, any $\E_n(R)$-conjugate of $\hat\sigma$ or $\hat\sigma^{-1}$ is an $\E_n(R)$-conjugate of $\sigma$ or $\sigma^{-1}$. Hence $t_{kl}(a\sigma'_{1n}b)$ is a product of $16$ $\E_n(R)$-conjugates of $\sigma$ and $\sigma^{-1}$. Assertion (i) now follows from Lemma \ref{lemp} (after swapping the roles of $\sigma$ and $\sigma^{-1}$). 
\item See the proof of Theorem \ref{thmNeum}.
\end{enumerate}
\end{proof}

\section{RDU over rings with Euclidean algorithm}

We call $R$ a {\it ring with $m$-term Euclidean algorithm} if for any row vector $v\in {}^m\!R$ there is a $\tau\in \E_m(R)$ such that $v\tau$ has a zero entry. Note that the rings with $2$-term Euclidean algorithm are precisely the right quasi-Euclidean rings defined in \cite{alahmadi} (follows from \cite[Theorem 11]{alahmadi}).

If $\tau\in \GL_m(R)$ for some $1\leq m<n$, then we identify $\tau$ with its image in $\GL_n(R)$ under the embedding 
\begin{align*}
\GL_m(R)&\hookrightarrow \GL_n(R)\\
\sigma&\mapsto\begin{pmatrix}e_{(n-m)\times(n-m)}&0\\0&\sigma\end{pmatrix}
\end{align*}

\begin{lemma}\label{lemeucl}
Let $\sigma\in \GL_n(R)$ and suppose $(\sigma\tau)_{1n}=0$ for some $\tau\in E_m(R)$ where $1\leq m<n$. Then for any $k\neq l$ and $a,b\in R$,
\begin{enumerate}[(i)]
\itemsep0pt 
\item $t_{kl}(a\sigma'_{1j}b)$ is a product of $8$ $\E_n(R)$-conjugates of $\sigma$ and $\sigma^{-1}$ if $j\in \{2,\dots,n-m\}$ and
\item $t_{kl}(a\sigma'_{1j}b)$ is a product of $8m$ $\E_n(R)$-conjugates of $\sigma$ and $\sigma^{-1}$ if $j\in \{n-m+1,\dots,n\}$.
\end{enumerate}
\end{lemma}
\begin{proof}
Set $\xi:=\sigma^{\tau}$. Then $\xi_{1n}=0$. It follows from Corollary \ref{corkeyB} that for any $2\leq j\leq n$, $k\neq l$ and $a,b\in R$, 
\begin{equation}
t_{kl}(a\xi'_{1j}b)\text{ is a product of }8~ \E_n(R)\text{-conjugates of }\sigma\text{ and }\sigma^{-1}.
\end{equation}
Clearly $\sigma^{-1}=(\xi^{-1})^{\tau^{-1}}$ and hence
\begin{equation}
\sigma'_{1j}=\begin{cases}
\xi'_{1j},&\text{ if } j\in \{2,\dots,n-m\},\\
\sum\limits_{p=n-m+1}^{n}\xi'_{1p}\tau'_{pj},&\text{ if } j\in \{n-m+1,\dots,n\}.
\end{cases}
\end{equation}
The assertion of the theorem follows from (7) and (8).
\end{proof}

\begin{theorem}\label{thmeucl_1}
Suppose that $R$ is a ring with $m$-term Euclidean algorithm for some $1\leq m\leq n-2$. Let $\sigma\in \GL_n(R)$, $i\neq j$, $k\neq l$ and $a,b,c\in R$. Then 
\begin{enumerate}[(i)]
\itemsep0pt 
\item $t_{kl}(a\sigma_{ij}b)$ is a product of $8$ $\E_n(R)$-conjugates of $\sigma$ and $\sigma^{-1}$ and
\item $t_{kl}(a(c\sigma_{ii}-\sigma_{jj}c)b)$ is a product of $24$ $\E_n(R)$-conjugates of $\sigma$ and $\sigma^{-1}$.
\end{enumerate}
\end{theorem}
\begin{proof}
\begin{enumerate}[(i)]
\item Since $R$ is a ring with $m$-term Euclidean algorithm, there is a $\tau\in \E_m(R)$ such that $(\sigma\tau)_{1n}=0$. It follows from Lemma \ref{lemeucl} that $t_{kl}(a\sigma'_{12}b)$ is a product of $8$ $\E_n(R)$-conjugates of $\sigma$ and $\sigma^{-1}$ (note that $n-m\geq 2$). Assertion (i) now follows from Lemma \ref{lemp} (after swapping the roles of $\sigma$ and $\sigma^{-1}$). 
\item See the proof of Theorem \ref{thmNeum}.
\end{enumerate}
\end{proof}

\begin{theorem}\label{thmeucl_2}
Suppose that $R$ is a ring with $n-1$-term Euclidean algorithm. Let $\sigma\in \GL_n(R)$, $i\neq j$, $k\neq l$ and $a,b,c\in R$. Then 
\begin{enumerate}[(i)]
\itemsep0pt 
\item $t_{kl}(a\sigma_{ij}b)$ is a product of $8(n-1)$ $\E_n(R)$-conjugates of $\sigma$ and $\sigma^{-1}$ and
\item $t_{kl}(a(c\sigma_{ii}-\sigma_{jj}c)b)$ is a product of $24(n-1)$ $\E_n(R)$-conjugates of $\sigma$ and $\sigma^{-1}$.
\end{enumerate}
\end{theorem}
\begin{proof}
\begin{enumerate}[(i)]
\item Since $R$ is a ring with $n-1$-term Euclidean algorithm, there is a $\tau\in \E_{n-1}(R)$ such that $(\sigma\tau)_{1n}=0$. It follows from Lemma \ref{lemeucl} that $t_{kl}(a\sigma'_{12}b)$ is a product of $8(n-1)$ $\E_n(R)$-conjugates of $\sigma$ and $\sigma^{-1}$. Assertion (i) now follows from Lemma \ref{lemp} (after swapping the roles of $\sigma$ and $\sigma^{-1}$). 
\item See the proof of Theorem \ref{thmNeum}.
\end{enumerate}
\end{proof}

Theorems \ref{thmeucl_1} and \ref{thmeucl_2} show that if $\sigma\in \GL_n(R)$ where $R$ is a ring with $m$-term Euclidean algorithm for some $1\leq m\leq n-1$, then the nondiagonal entries of $\sigma$ can be ``extracted'' (i.e. the matrices $t_{kl}(\sigma_{ij})~(k\neq l, i\neq j)$ lie in $\sigma^{\E_n(R)}$). Does this also hold for $m=n$? The author does not know the answer to this question. However, if $R$ is a ring with {\it strong} $n$-term Euclidean algorithm (see Definition \ref{defstrong} below), then the nondiagonal entries of $\sigma$ can be extracted, as Theorem \ref{thmeucl_3} shows.

\begin{definition}\label{defstrong}
We call $R$ a {\it ring with strong $m$-term Euclidean algorithm} if for any row vector $v\in {}^m\!R$ there is a $\tau\in \E_m(R)$ with $\tau_{11}=1$ such that $(v\tau)_{n}=0$.
\end{definition}

\begin{remark}
If $R$ is a ring with strong $m$-term Euclidean algorithm, then obviously $R$ is also a ring with $m$-term Euclidean algorithm.
\end{remark}

\begin{theorem}\label{thmeucl_3}
Suppose that $R$ is a ring with strong $n$-term Euclidean algorithm. Let $\sigma\in \GL_n(R)$, $i\neq j$, $k\neq l$ and $a,b,c\in R$. Then 
\begin{enumerate}[(i)]
\itemsep0pt 
\item $t_{kl}(a\sigma_{ij}b)$ is a product of $80(n-1)$ $\E_n(R)$-conjugates of $\sigma$ and $\sigma^{-1}$ and
\item $t_{kl}(a(c\sigma_{ii}-\sigma_{jj}c)b)$ is a product of $240(n-1)$ $\E_n(R)$-conjugates of $\sigma$ and $\sigma^{-1}$.
\end{enumerate}
\end{theorem}
\begin{proof}
\begin{enumerate}[(i)]
\item Since $R$ is a ring with strong $n$-term Euclidean algorithm, there is a $\tau\in \E_{n}(R)$ with $\tau_{11}=1$ such that $(\sigma\tau)_{1n}=0$. Clearly the matrix $\tau\prod\limits_{j=2}^{n-1}t_{1j}(-\tau_{1j})$ has the same properties as $\tau$. Hence we may assume that $\tau_{1j}=0~(j=2,\dots,n-1)$. \\
{\bf Step 1} 
Clearly $\tau=\rho t_{1n}(\tau_{1n})$ for some $\rho\in \E_n(R)$ with trivial first row. Set $\xi:=t_{1n}(\tau_{1n})\sigma^{\tau}=\rho^{-1}\sigma\tau$. Then $\xi_{1n}=0$. One checks easily that 
\[(t_{1n}(-\tau_{1n}),\xi)\xrightarrow{t_{n2}(b'),t_{31}(a'),t_{21}(-1)}(t_{31}(a'\tau_{1n}b'), e)\]
for any $a',b'\in R$. It follows from Lemma \ref{lemp} and Lemma \ref{lemredux} that 
\begin{equation}
t_{k'l'}(a'\tau_{1n}b')\text{ is a product of }8~\E_n(R)\text{-conjugates of }\sigma\text{ and }\sigma^{-1}
\end{equation}
for any $k'\neq l'$ and $a',b'\in R$.\\
{\bf Step 2} Let $\xi$ be defined as in Step 1. Corollary \ref{corkeyB} implies that for any $2\leq j\leq n$, $k''\neq l''$ and $a'',b''\in R$, the matrix $t_{k''l''}(a''\xi'_{1j}b'')$ is a product of $8$ $\E_n(R)$-conjugates of $\xi$ and $\xi^{-1}$. Hence, by (9),
\begin{equation}
t_{k''l''}(a''\xi'_{1j}b'')\text{ is a product of }8\cdot 9=72~ \E_n(R)\text{-conjugates of }\sigma\text{ and }\sigma^{-1}.
\end{equation}
{\bf Step 3} Clearly $\sigma^{-1}=\tau\xi^{-1}\rho^{-1}$ and hence $\sigma'_{12}=\sum\limits_{j=2}^{n}(\xi'_{1j}+\tau_{1n}\xi'_{nj})\rho'_{j2}$. It follows from (9) and (10) that $t_{kl}(a\sigma'_{12}b)$ is a product of $(72+8)(n-1)=80(n-1)$ $\E_n(R)$-conjugates of $\sigma$ and $\sigma^{-1}$. Assertion (i) now follows from Lemma \ref{lemp} (after swapping the roles of $\sigma$ and $\sigma^{-1}$). 
\item See the proof of Theorem \ref{thmNeum}.
\end{enumerate}
\end{proof}

\section{RDU over almost commutative rings}
In this subsection $C$ denotes the center of $R$.

\begin{definition}
Let $x\in R$. An element $z\in C$ is called a {\it central multiple} of $x$ if $z=xy=yx$ for some $y\in R$. The ideal of $C$ consisting of all central multiples of $x$ is denoted by $\Oo(x)$.
\end{definition}

\begin{lemma}\label{lemalmcom1}
Let $\sigma\in \GL_n(R)$ and $z\in \Oo(\sigma_{11})$. Then for any $k\neq l$ and $a,b\in R$, $t_{kl}(az\sigma_{12}b)$ is a product of of $8$ $\E_n(R)$-conjugates of $\sigma$ and $\sigma^{-1}$.
\end{lemma}
\begin{proof}
Since $z\in \Oo(\sigma_{11})$, there is a $y\in R$ such that $z=\sigma_{11}y=y\sigma_{11}$. Clearly $\sigma_{11}y\sigma_{12}-\sigma_{12}\sigma_{11}y=0$. It follows from Proposition \ref{propkey}(i) that $t_{kl}(az\sigma_{12}b)$ is a product of $8$ $\E_n(R)$-conjugates of $\sigma$ and $\sigma^{-1}$.
\end{proof}

Set $\E^{**}_n(R):=\langle t_{k1}(x)\mid x\in R,k\neq 1\rangle$.

\begin{lemma}\label{lemalmcom2}
Let $\sigma\in \GL_n(R)$. Suppose that there are $\tau_p\in E^{**}_n(R)~(1\leq p\leq q)$ and $z_p\in \Oo((\sigma^{\tau_{_p}})_{11})~(1\leq p\leq q)$ such that $\sum\limits_{p=1}^{q}z_p=1$. Then for any $k\neq l$ and $a,b\in R$, $t_{kl}(a\sigma_{12}b)$ is a product of of $8q$ $\E_n(R)$-conjugates of $\sigma$ and $\sigma^{-1}$. 
\end{lemma}
\begin{proof}
The previous lemma implies that for any $1\leq p \leq q$, the elementary transvection $t_{kl}(az_p(\sigma^{\tau_{p}})_{12}b)=t_{kl}(az_p\sigma_{12}b)$ (for this equation we have used that $\tau_{p}\in E^{**}_n(R)$) is a product of $8$ elementary $\sigma$-conjugates (note that any elementary $\sigma^{\tau_{p}}$-conjugate is also an elementary $\sigma$-conjugate). It follows that $t_{kl}(a\sigma_{12}b)=t_{kl}(az_1\sigma_{12}b)\dots t_{kl}(az_q\sigma_{12}b)$ is a product of $8q$ elementary $\sigma$-conjugates. 
\end{proof}

Recall that $R$ is called {\it almost commutative} if it is finitely generated as a $C$-module. We will show that if $R$ is almost commutative, then the requirements of Lemma \ref{lemalmcom2} are always satisfied. 

We denote by $\Max(C)$ the set of all maximal ideals of $C$. If $\m\in \Max(C)$, then we denote by $R_{\m}$ the localisation of $R$ with respect to the multiplicative set $S_{\m}:= C\setminus \m$. We denote by $Q_{\m}$ the quotient $R_{\m}/\Rad(R_{\m})$ where $\Rad(R_{\m})$ is the Jacobson radical of $R_{\m}$. Moreover, we denote the canonical ring homomorphism $R\rightarrow Q_{\m}$ by $\phi_{\m}$. 

\begin{lemma}\label{lemalmcom3}
If $R$ is almost commutative, then for any $\m\in \Max(C)$, $\phi_{\m}$ is surjective and $\sr(Q_{\m})=1$.
\end{lemma}
\begin{proof}
First we show that $\phi_\m$ is surjective. Set $C_\m:=S_\m^{-1}C$. Clearly $R_\m$ is finitely generated as a $C_\m$-module. Hence, by \cite[Corollary 5.9]{lam_book}, we have $\m=\Rad(C_\m)\subseteq \Rad (R_\m)$. Let now $\frac{r}{s}\in R_\m$. Since $sC+\m=C$, there is a $c\in C$ and an $m\in\m$ such that $sc+m=1$. Multiplying this equality by $\frac{r}{s}$ we get $\frac{r}{s}=rc+\frac{rm}{s}\in rc+\Rad(R_\m)$. Thus $\phi_{\m}$ is surjective.\\
Next we show that $\sr(Q_{\m})=1$. Clearly $Q_\m=R_\m/\Rad(R_{\m})$ is finitely generated as a $C_\m/\Rad(C_{\m})$-module. But $C_\m/\Rad(C_{\m})$ is a field. Hence $Q_\m$ is a semisimple ring (see \cite[\S 7]{lam_book}) and therefore $\sr(Q_{\m})=1$ (see \cite[\S6]{bass}).
\end{proof}

\begin{lemma}\label{lemalmcom4}
Suppose that $R$ is almost commutative. Let $u=(u_1,\dots,u_n)\in {}^n\!R$ be a unimodular row. Then there are matrices $\tau_{\m}\in \E^{**}_n(R)~(\m\in \Max(C))$ such that $C=\sum\limits_{\m\in \Max(C)}\Oo((u\tau_{\m})_1)$.
\end{lemma}
\begin{proof}
Let $\m\in\Max(C)$ and set $\hat u:=\phi_\m(u)\in {}^n\!(Q_m)$. Since $\sr(Q_{\m})=1$ and $\hat u$ is unimodular, there is a $\hat\tau_{\m}\in \E_n(Q_{\m})$ such that $(\hat u\hat\tau_{\m})_1$ is invertible (note that rings of stable rank $1$ are Dedekind finite, see e.g. \cite[Lemma 1.7]{lam}). Clearly $\hat\tau_{\m}$ can be chosen to be in $E_n^{**}(Q_\m)$ (follows from \cite[Theorem 1]{vaserstein_sr}). Since $\phi_{\m}$ is surjective, it induces a surjective homomorphism $\E^{**}_n(R)\rightarrow\E^{**}_n(Q_{\m})$ which we also denote by $\phi_\m$. Choose a $\tau_\m\in E^{**}_n(R)$ such that $\phi_\m(\tau_\m)=\hat\tau_\m$.\\
Since $\phi_{\m}((u\tau_{\m})_1)=(\hat u\hat\tau_{\m})_1$ is invertible in $Q_\m$, $(u\tau_{\m})_1$ is invertible in $R_{\m}$. Write $((u\tau_{\m})_1)^{-1}=\frac{x}{s}$ where $x\in R$ and $s\in S_{\m}$. Then
\begin{align*}
\frac{(u\tau_{\m})_1}{1}\frac{x}{s}=\frac{1}{1}~\Leftrightarrow~\exists t\in S_{\m}:t(u\tau_{\m})_1x=ts
\end{align*}
and
\begin{align*}
\frac{x}{s}\frac{(u\tau_{\m})_1}{1}=\frac{1}{1}~\Leftrightarrow~\exists t'\in S_{\m}:t'x(u\tau_{\m})_1=t's.
\end{align*}
Clearly $z_{\m}:=(u\tau_{\m})_1tt'x=tt's=tt'x(u\tau_{\m})_1\in \Oo((u\tau_{\m})_1)\cap S_{\m}$. We have shown that for any $\m\in\Max(C)$ there is a $\tau_{\m}\in \E^{**}_n(R)$ and a $z_{\m}\in \Oo((u\tau_{\m})_1)$ such that $z_{\m}\not\in \m$. The assertion of the lemma follows.
\end{proof}

\begin{corollary}\label{coralmcom}
Suppose that $R$ is almost commutative and let $\sigma\in \GL_n(R)$. Then there is a $q\leq |\Max(C)|$, $\tau_p\in E^{**}_n(R)~(1\leq p\leq q)$ and $z_p\in \Oo((\sigma^{\tau_{_p}})_{11})~(1\leq p\leq q)$ such that $\sum\limits_{p=1}^{q}z_p=1$.
\end{corollary}
\begin{proof}
By Lemma \ref{lemalmcom4} there are matrices $\tau_{\m}\in E^{**}_n(R)~(\m\in \Max(C))$ such that 
\[C=\sum\limits_{\m\in \Max(C)}\Oo((\sigma_{1*}\tau_{\m})_1)=\sum\limits_{\m\in \Max(C)}\Oo((\sigma\tau_{\m})_{11})=\sum\limits_{\m\in \Max(C)}\Oo((\sigma^{\tau_{\m}})_{11})\]
(for the last equation we have used that $\tau_{\m}\in E^{**}_n(R)$). The assertion of the corollary follows.
\end{proof}

\begin{theorem}\label{thmalmcom}
Suppose that $R$ is almost commutative. Let $\sigma\in \GL_n(R)$, $i\neq j$, $k\neq l$ and $a,b,c\in R$. Then 
\begin{enumerate}[(i)]
\itemsep0pt 
\item $t_{kl}(a\sigma_{ij}b)$ is a finite product of $8|\Max(C)|$ or less $\E_n(R)$-conjugates of $\sigma$ and $\sigma^{-1}$ and
\item $t_{kl}(a(c\sigma_{ii}-\sigma_{jj}c)b)$ is a finite product of $24|\Max(C)|$ or less $\E_n(R)$-conjugates of $\sigma$ and $\sigma^{-1}$.
\end{enumerate}
\end{theorem}
\begin{proof}
\begin{enumerate}[(i)]
\item It follows from Lemma \ref{lemalmcom2} and Corollary \ref{coralmcom} that that $t_{kl}(a\sigma_{12}b)$ is a product of $8q$ elementary $\sigma$-conjugates for some $q\leq|\Max(C)|$. Assertion (i) now follows from Lemma \ref{lemp}.
\item See the proof of Theorem \ref{thmNeum}.
\end{enumerate}
\end{proof}

\section{Open problems}
\begin{definition}
Let $m\in\N$. Then $m$ is called an {\it RDU bound for $\GL_n(R)$} if for any $\sigma\in\GL_n(R)$, $k\neq l$ and $i\neq j$ the matrix $t_{kl}(\sigma_{ij})$ is a product of $m$ or less $\E_n(R)$-conjugates of $\sigma$ and $\sigma^{-1}$. If $m$ is an RDU bound for $\GL_n(R)$ for any $n\geq 3$, then $m$ is called an {\it RDU bound for $R$}. If $\Cl$ is a class of rings and $m$ is an RDU bound for $R$ for any $R\in \Cl$, then $m$ is called an {\it RDU bound for $\Cl$}.
\end{definition}

It follows from Sections 5, 6 and 8 that $8$ is an RDU bound for the class of commutative rings, the class of von Neumann regular rings, and the class of rings of stable rank $1$. It follows from Sections 7 and 9 that $160$ is an RDU bound for the class of Banach algebras and $16$ is an RDU bound for the class of quasi-Euclidean rings. Note that Section 10 does not yield an RDU bound for the class of almost commutative rings since the maximal spectrum of an almost commutative ring might be infinite.

It is natural to ask the following two question.

\begin{question}\label{Q1}
Which rings have an RDU bound? 
\end{question}

\begin{question}\label{Q2}
What is the optimal (i.e. smallest) RDU bound for the class of commutative rings (resp. von Neumann regular rings, rings of stable rank $1$ etc.)? 
\end{question}

Using a computer program the author found out that $2$ is the optimal RDU bound for $\GL_3(\F_2)$ and $\GL_3(\F_3)$ where $\F_2$ and $\F_3$ denote the fields of $2$ and $3$ elements, respectively. He does not know the optimal bounds in any other cases.

\begin{lemma}
Suppose that $1$ is an RDU bound for a ring $R$. Then $\GL_n(R)=\E_n(R)$ for any $n\geq 3$.
\end{lemma}
\begin{proof}
If $n\geq 3$ and $\sigma\in \GL_n(R)$, then, by the definition of an RDU bound, $\sigma=t_{12}(\sigma_{12})^{\tau}$ or $\sigma^{-1}=t_{12}(\sigma_{12})^{\tau}$ for some $\tau\in \E_n(R)$. Hence $\sigma\in \E_n(R)$.
\end{proof}

Since $\E_n(R)\subseteq\SL_n(R)\neq \GL_n(R)$ if $R$ is a field with more than two elements, we get the following result. 

\begin{theorem}
The optimal RDU bound for the class of commutative rings (resp. von Neumann regular rings, rings of stable rank $1$) lies between $2$ and $8$.
\end{theorem}


The author thinks that the following question is interesting too (compare Lemma \ref{lemeucl}).

\begin{question}\label{Q3}
Let $\sigma\in \GL_n(R)$ and suppose $(\sigma\tau)_{1n}=0$ for some $\tau\in E_n(R)$. Can one extract some nondiagonal entry of $\sigma$ or $\sigma^{-1}$? (I.e., is it true that $t_{kl}(\sigma_{ij})\in \sigma^{\E_n(R)}$ or $t_{kl}(\sigma'_{ij})\in \sigma^{\E_n(R)}$ for some $k\neq l$ and $i\neq j$?)
\end{question}

A positive answer to Question \ref{Q3} would have the interesting consequence that for any $\sigma$ lying in the elementary subgroup $\E_n(R)$, one could extract some nondiagonal entry of $\sigma$ or $\sigma^{-1}$ (since $(\sigma\sigma^{-1})_{1n}=0$).

\end{document}